\documentclass{amsart}

\usepackage{amsfonts}
\usepackage{amssymb}
\usepackage{amsthm}
\usepackage{graphicx}

\newtheorem{lemma}{Proposition}
\newtheorem{corollary}[lemma]{Corollary}

\begin{document}

\title[Two energy-preserving finite-difference schemes]{Two finite-difference schemes that preserve the dissipation of energy in a system of modified wave equations}

\author{J. E. Mac\'{\i}as-D\'{\i}az}
\address{Departamento de Matem\'{a}ticas y F\'{\i}sica, Universidad Aut\'{o}noma de Aguascalientes, Avenida Universidad 940, Ciudad Universitaria, Aguascalientes, Ags. 20100, Mexico}
\email{jemacias@correo.uaa.mx}

\author{S. Jerez-Galiano}
\address[CIMAT]{Centro de Investigaci\'{o}n en Matem\'{a}ticas, A. C., Jalisco S/N, Colonia Valenciana, Guanajuato, Gto. 36240, Mexico}
\email{jerez@cimat.mx}


\begin{abstract}
In this work, we present two numerical methods to approximate solutions of systems of dissipative sine-Gordon equations that arise in the study of one-dimensional, semi-infinite arrays of Josephson junctions coupled through superconducting wires. Also, we present schemes for the total energy of such systems in association with the finite-difference schemes used to approximate the solutions. The proposed methods are conditionally stable techniques that yield consistent approximations not only in the domains of the solution and the total energy, but also in the approximation to the rate of change of energy with respect to time. The methods are employed in the estimation of the threshold at which nonlinear supratransmission takes place, in the presence of parameters such as internal and external damping, generalized mass, and generalized Josephson current. Our results are qualitatively in agreement with the corresponding problem in mechanical chains of coupled oscillators, under the presence of the same parameters.
\end{abstract}

\maketitle

\section{Introduction}

The well-known sine-Gordon equation is an important differential equation in physics due to its multiple applications. For example, it appears in the study of long Josephson junctions between superconductors \cite{Solitons}, in the investigation of discrete arrays of Josephson junctions coupled through superconducting wires \cite {Chevrieux2}, in the modeling of the motion of a damped string in a non-Hookean medium or, more generally, in the description of the motion of rigid pendula attached to a stretched wire \cite {Barone}, and as a model for rapidly rotating baroclinic fluids \cite {Gibbon}. Modified versions of the sine-Gordon equation often appear in the form of Klein-Gordon equations, in problems such as the study of fluxons in Josephson transmission lines \cite{Lomdahl}, in the statistical mechanics of nonlinear coherent structures such as solitary waves (see \cite{Makhankov} pp. 298--309), in the study of Alfven waves in nonuniform media \cite {Musielak}, or as Taylor approximations to problems that involve the sine-Gordon model.

The continuous form of the sine-Gordon equation possesses many interesting analytical properties. For example, the sine-Gordon equation is an integrable equation, a characteristic which does not share, for instance, with the nonlinear Klein-Gordon equation. The existence of traveling-wave solutions, soliton solutions and, moreover, nonlinear intrinsic modes (called \emph {moving breathers}), are also well-known facts associated to this equation \cite {handbook}. In summary, many interesting properties are already available for the analysis of continuous sine-Gordon systems; however, the theory underneath the foundations of this differential equation is rather far from being completely understood, whence it follows that the determination of novel results related to continuous sine-Gordon media, as well as the study of new physical applications, are highly transited highways in the literature of applied mathematics and physics.

From a numerical perspective, the study of the classical, continuous sine-Gordon equation has been a popular source of interest since its inception \cite {Perring}. For instance, Josephson tunnel junctions were studied numerically via the sine-Gordon equation in \cite{Lomdahl}, while the numerical study of wave transmission in long Josephson structures subject to harmonic driving was begun by Olsen and Samuelsen \cite{Olsen}. On the other hand, the $(1 + 1)$-dimensional Langevin equation was solved through numerical simulations in \cite{AlexHabib}. Later on, Strauss and V\'{a}z\-quez designed a numerical technique to approximate radially symmetric solutions of conservative Klein-Gordon equations \cite{StraussVazquez}, one of the most important features of their numerical method being that the discrete energy associated with the differential equation is conserved.

It must be mentioned that, after Strauss and V\'{a}zquez's pioneering work, the development of numerical techniques with properties of invariance in some physical domains became a very fruitful topic of research. Particularly, many finite-difference schemes (implicit and explicit) with properties of invariance in the energy domain were proposed for classes of nonlinear, conservative Klein-Gordon equations \cite {StraussVazquez,Ben,Fei,Li}. However, it is worth mentioning that the nonconservative case has been left aside, and that its study has been recently initiated \cite {Macias-Puri,Macias-Supra,Macias2D,Macias-CNSNS}, motivated by the study of the nonlinear phenomenon of supratransmission of energy in certain dissipative systems, where the sine-Gordon equation clearly dominates the theoretical \cite {Chevrieux2,Geniet-Leon,Geniet-Leon2} and the applied \cite{Macias-Signals,Macias-Signals2} scenarios.

This article presents two conditionally stable, finite-difference schemes with properties of consistency in the domains of the solution and the energy. The numerical techniques are used in the investigation of the process of nonlinear supratransmission in discrete arrays of Josephson junctions subject to harmonic perturbations in one end, for which the methods presented here seem to be natural choices in the study of the problem.

Section \ref {Sec2} of this paper introduces the mathematical model under study and the energy expressions employed in the analysis of supratransmission. Section \ref {Sec3} presents two numerical schemes associated to our problem and some practical observations on the implementation of the methods. The next section establishes the most important numerical properties of the computational techniques; here, we show that the proposed methods provide consistent approximations in the energy domain --- a highly desirable characteristic in the analysis of the process of supratransmission ---, and establish stability properties. In Section \ref {Sec5}, our techniques are applied to the determination of the supratransmission threshold in the presence of several parameters. Finally, the article closes with a section of concluding remarks and directions of further research.

\section{Preliminaries\label{Sec2}}

\subsection{Mathematical model}

The physical motivation of the present work is the study of discrete, linear arrays of parallel Josephson junctions coupled through superconducting wires, subject to harmonic driving on one end. The model considers the effects of internal and external damping, relativistic mass, and a generalized Josephson current. More concretely, we consider a sequence of $N$ coupled Josephson junctions initially at rest, and study the associated initial-value problem
\begin{equation}
\begin{array}{c}
\begin{array}{rcl}
\displaystyle {\ddot {u} _1 - \left( c ^2 + \beta D _t \right) \Delta ^2 _x u _1 + \gamma \dot {u} _1 + \mathfrak {m} ^2 u _1 + \sin u _1} & = & J + \phi (t), \\
\displaystyle {\ddot {u} _n - \left( c ^2 + \beta D _t \right) \Delta ^2 _x u _n + \gamma \dot {u} _n + \mathfrak {m} ^2 u _n + \sin u _n} & = & J, \ \ 1 < n < N, \\ 
\displaystyle {\ddot {u} _N + \left( c ^2 + \beta D _t \right) \Delta ^2 _x u _{N - 1} + \gamma \dot {u} _N + \mathfrak {m} ^2 u _N + \sin u _N} & = & J -I, 
\end{array}\\
		\begin{array}{rl}
        \begin{array}{l}
            \textrm {s.\ t.}
        \end{array}
        \left\{
        \begin{array}{ll}
            u _n (0) = 0, & 1 \leq n \leq N, \\
            \displaystyle {\frac {d u _n} {d t} (0) = 0}, & 1 \leq n \leq N. \\
        \end{array}\right.
    \end{array}
\end{array}\label{Eqn:Main1}
\end{equation}
Here, $D _t$ is the derivative operator with respect to time $t$, 
\begin{equation}
\Delta _x ^2 u _n = (u _{n + 1} - u _n) \chi _{[1 , N]} (n + 1) + (u _{n - 1} - u _n) \chi _{[1 , N]} (n - 1),
\end{equation}
and $\chi _{[1 , N]} (n)$ is the characteristic function on the set $[1 , N]$ evaluated at $n$, which is equal to $1$ if $n$ belongs to $[1 , N]$, and equal to $0$ otherwise.

Notationally, $c$ is a nonnegative constant called the \emph {coupling coefficient}, and $\beta$ and $\gamma$ are also nonnegative constants called the \emph {internal} and \emph {external damping coefficients}, respectively. The constant $\mathfrak {m}$ --- called the \emph {mass} of the system --- is a real or pure-imaginary complex number that has been included in the system of equations in order to suggest possible, further applications of our technique to phi-four theories and super symmetry \cite {Kudriatsev}. Meantime, the function $\phi$ is called the \emph {input intensity function}, and it is a harmonic disturbance irradiating at a frequency in the forbidden band-gap of the system, that takes the concrete form $\phi (t) = A \sin (\Omega t)$. Finally, the constants $J$ and $I$ are called, respectively, the \emph {generalized Josephson current} and the \emph {output current intensity} of the system.

It has been established \cite {Macias-Signals3} that the inclusion of two convenient nodes, one located at the beginning and the other at the end of the array of the $N$ Josephson junctions in (\ref {Eqn:Main1}), transforms our initial-value problem into the equivalent initial-value problem with boundary data
\begin{equation}
\begin{array}{c}
\displaystyle {\ddot {u} _n - \left( c ^2 + \beta D _t \right) \Delta _x ^2 u _n + \gamma _n \dot {u} _n + \mathfrak {m} ^2 u _n + \sin u _n} = J, \ \ \ 1 \leq n \leq N, \\ 
		\begin{array}{rl}
        \begin{array}{l}
            \textrm {s.\ t.}
        \end{array}
        \left\{
        \begin{array}{ll}
            u _n (0) = 0, \quad \dot {u} _n (0) = 0, & 1 \leq n \leq N, \\
            c ^2 (u _0 - u _1) + \beta (\dot {u} _0 - \dot {u} _1) = \phi, & t \geq 0, \\
            u _{N + 1} - u _N = 0, & t \geq 0,
        \end{array}\right.
    \end{array}
\end{array}\label{Eqn:Main2}
\end{equation}
where $\gamma _n = \gamma$ for every $n < N$, and $\gamma _N = \gamma + 1 / R$. The number $R$ is called the \emph {output reading resistance} of the system, and it is related to the output current intensity through Ohm's law: $I = \dot {u} _N / R$.

Let us consider the case $\beta = \gamma = J = 0$. Then, the linear dispersion relation associated with the linearized form of the differential equations in (\ref {Eqn:Main2}) is given by $\omega ^2 (k) = \mathfrak {m} ^2 + 1 + 2 c ^2 (1 - \cos k)$. It is clear that $1 - \cos k \geq 0$ for every $k$, so that $\omega ^2 (k) \geq \mathfrak {m} ^2 + 1$. It follows that the forbidden band-gap is provided by the inequality $\Omega < \sqrt {\mathfrak {m} ^2 + 1}$.

\subsection{Energy expressions}

It is worth noticing that the Hamiltonian of the $n$-site in the conservative version of (\ref {Eqn:Main2}) is given by
\begin{equation}
H _n = \frac {1} {2} \left[ \dot {u} _n ^2 + c ^2 (u _{n + 1} - u _n) ^2 + \mathfrak {m} ^2 u _n ^2 \right] + V (u _n) - J u _n,
\end{equation}
with $V (u _n)$ being the potential function for the classical sine-Gordon equation evaluated at $u _n$, namely, $V (u _n) = 1 - \cos u _n$. The inclusion of the potential between the coupling of the first two sites in the chain of Josephson junctions results in the following expression for the total energy of the system:
\begin{equation}
E = \sum _{n = 1} ^N H _n + \frac {c ^2} {2} (u _1 - u _0) ^2.
\end{equation}

As it is customary in the analysis of supratransmission, it is convenient to consider semi-infinite, discrete arrays of nodes, in order to determine the critical threshold at which the system under study starts to propagate wave signals in the form of localized modes. To do that, one may eventually need to consider infinite sequences $(u _n) _{n = 1} ^\infty$ of real functions satisfying the differential equations of problem (\ref {Eqn:Main2}), for which the sequences of their corresponding derivatives are members of $\ell ^2 (\mathbb {R})$. More precisely, one needs to consider sequences $(u _n) _{n = 1} ^\infty$ for which $\sum _{n = 1} ^\infty (\dot {u} _n) ^2$ is convergent at any time. This type of sequences will be called \emph {square-summable}. 

\begin{lemma}[Mac\'{\i}as-D\'{\i}az and Puri \cite {Macias-Supra}]
Let $(u _n) _{n = 1} ^N$ be a solution of the system of differential equations in (\ref {Eqn:Main2}). Then, the instantaneous rate of change of the total energy of the system with respect to time is given by
\begin{equation}
\frac {d E} {d t} = c ^2 (u _0 - u _1) \dot {u} _0 - \beta \left[ \sum _{n = 1} ^N (\dot {u} _n - \dot {u} _{n - 1}) ^2 + (\dot {u} _1 - \dot {u} _0) \dot {u} _0 \right] - \gamma \sum _{n = 1} ^N \dot {u} _n ^2. \label{Eqn:Hola}
\end{equation}
Moreover, if $(u _n) _{n = 1} ^\infty$ is a square-summable solution of an infinite class of coupled differential equations described by (\ref {Eqn:Main2}), then the instantaneous rate of change of total energy of the system is obtained by taking the limit when $N$ tends to infinity in the right-hand side of (\ref {Eqn:Hola}). \qed
\end{lemma}

\begin{corollary} [Geniet and Leon \cite{Geniet-Leon2}]
Let $(u _n) _{n = 1} ^N$ be a solution of the undamped version of (\ref {Eqn:Main2}). Then, the total energy of the system is given by
\begin{equation}
E (t) = - c ^2 \int _0 ^t \dot {u} _0 (s) (u _1 (s) - u _0 (s)) \ ds.
\end{equation} \qed
\end{corollary}

\section{Computational techniques\label {Sec3}}

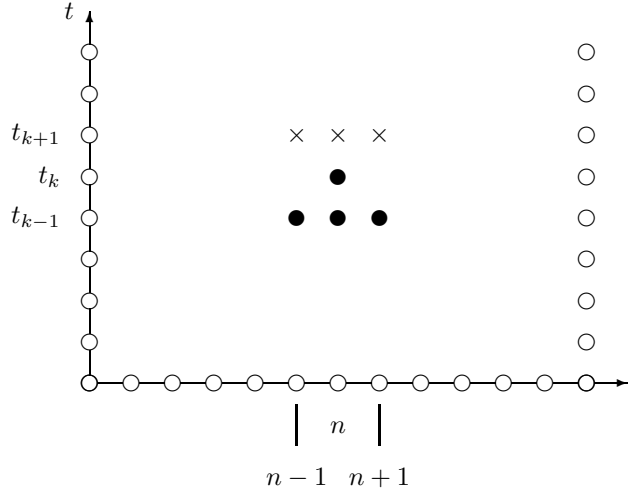
\begin{figure}
\begin{center}
\setlength{\unitlength}{.55mm}
\begin{picture}(140,100)
\put(10,12){\line(0,1){6}} %
\put(10,22){\line(0,1){6}} %
\put(10,32){\line(0,1){6}} %
\put(10,42){\line(0,1){6}} %
\put(10,52){\line(0,1){6}} %
\put(10,62){\line(0,1){6}} %
\put(10,72){\line(0,1){6}} %
\put(10,82){\line(0,1){6}} %
\put(10,92){\vector(0,1){8}} %
\put(5,100){\makebox(0,0)[5]{$t$}} %
\put(12,10){\line(1,0){6}} %
\put(22,10){\line(1,0){6}} %
\put(32,10){\line(1,0){6}} %
\put(42,10){\line(1,0){6}} %
\put(52,10){\line(1,0){6}} %
\put(62,10){\line(1,0){6}} %
\put(72,10){\line(1,0){6}} %
\put(82,10){\line(1,0){6}} %
\put(92,10){\line(1,0){6}} %
\put(102,10){\line(1,0){6}} %
\put(112,10){\line(1,0){6}} %
\put(122,10){\line(1,0){6}} %
\put(132,10){\vector(1,0){8}} %
\put(10,10){\circle{4}} %
\put(10,20){\circle{4}} %
\put(10,30){\circle{4}} %
\put(10,40){\circle{4}} %
\put(10,50){\circle{4}} %
\put(10,60){\circle{4}} %
\put(10,70){\circle{4}} %
\put(10,80){\circle{4}} %
\put(10,90){\circle{4}} %
\put(130,10){\circle{4}} %
\put(130,20){\circle{4}} %
\put(130,30){\circle{4}} %
\put(130,40){\circle{4}} %
\put(130,50){\circle{4}} %
\put(130,60){\circle{4}} %
\put(130,70){\circle{4}} %
\put(130,80){\circle{4}} %
\put(130,90){\circle{4}} %
\put(10,10){\circle{4}} %
\put(20,10){\circle{4}} %
\put(30,10){\circle{4}} %
\put(40,10){\circle{4}} %
\put(50,10){\circle{4}} %
\put(60,10){\circle{4}} %
\put(70,10){\circle{4}} %
\put(80,10){\circle{4}} %
\put(90,10){\circle{4}} %
\put(100,10){\circle{4}} %
\put(110,10){\circle{4}} %
\put(120,10){\circle{4}} %
\put(130,10){\circle{4}} %
\put(70,-2){\makebox(0,0)[b]{$n$}} %
\put(60,-15){\makebox(0,0)[b]{$n - 1$}} %
\put(60,-5){\line(0,1){10}} %
\put(80,-5){\line(0,1){10}} %
\put(80,-15){\makebox(0,0)[b]{$n + 1$}} %
\put(3,60){\makebox(0,0)[r]{$t _k$}} %
\put(3,70){\makebox(0,0)[r]{$t _{k + 1}$}} %
\put(3,50){\makebox(0,0)[r]{$t _{k - 1}$}} %
\put(70,60){\circle*{4}} %
\put(70,50){\circle*{4}} %
\put(60,50){\circle*{4}} %
\put(80,50){\circle*{4}} %
\put(60,70){\makebox(0,0){$\times$}} %
\put(70,70){\makebox(0,0){$\times$}} %
\put(80,70){\makebox(0,0){$\times$}} %
\end{picture} 
\end{center}\ \smallskip
\caption{Forward-difference stencil for the approximation to the $n$th differential equation in (\ref {Eqn:Main2}) at time $t _k$, using scheme (\ref {Eqn:DiscrMain}). The black circles represent known approximations to the actual solutions at times $t _{k - 1}$ and $t _k$ in scheme (\ref {Eqn:DiscrMain}), and the crosses denote the unknown approximations at time $t _{k + 1}$.  \label{Fig:Scheme}}
\end{figure}

\subsection{A finite-difference scheme}

Consider a class of $N$ differential equations satisfying (\ref {Eqn:Main2}), and let us take a regular partition $0 = t _0 < t _1 < \dots < t _M = T$ of a time interval $[0 , T]$, with time step equal to $\Delta t$. For each $k = 0 , 1 , \dots , M$, let us represent by $u ^k _n$ the approximate solution to our problem on the $n$th lattice site at time $t _k$, and the actual value of $\phi$ at the $k$th time step by $\phi _k$. Let us convey now that 
\begin{equation}
\begin{array}{rcl}
\delta _t u _n ^k & = & u _n ^{k + 1} - u _n ^{k - 1},\\ 
\delta ^2 _t u _n ^k & = & u _n ^{k + 1} - 2 u _n ^k + u _n ^{k - 1},\\ 
\delta ^2 _x u _n ^k & = & u _{n + 1} ^k - 2 u _n ^k + u _{n - 1} ^k,\\
\bar {\delta} u _n ^k & = & u _n ^{k + 1} + u _n ^{k - 1}.
\end{array}
\end{equation}
Then, problem (\ref {Eqn:Main2}) takes the discrete form
\begin{equation}
\begin{array}{c}
\begin{array}{rl}\left.
\begin{array}{rcl}
\displaystyle {\frac {\delta ^2 _t u _n ^k} {(\Delta t) ^2} - \left( \frac {c ^2} {2} \bar {\delta} + \frac {\beta} {2 \Delta t} \delta _t \right) \delta ^2 _x u _n ^k + \frac {\gamma} {2 \Delta t} \delta _t u _n ^k +} \quad & & \\
\displaystyle {\frac {\mathfrak {m} ^2} {2} \bar {\delta} u _n ^k + \frac {V (u _n ^{k + 1}) - V (u _n ^{k - 1})} {u _n ^{k + 1} - u _n ^{k - 1}}} & = & J,
\end{array} \right\} &
\begin{array}{l}
1 \leq n \leq N, \\
1 \leq k \leq M,
\end{array}
\end{array} \\
\textrm {s.\ t.} \left\{ \begin{array}{ll}
u _n ^0 = 0, \quad u _n ^1 = 0, & 1 \leq n \leq N, \\
c ^2 \bar {\delta} (u _0 ^k - u _1 ^k) + \beta \delta _t (u _0 ^k - u _1 ^k) / \Delta t = 2 \phi _k, & 1 \leq k \leq M, \\
u _{N + 1} ^k - u _N ^k = 0, & 1 \leq k \leq M.
\end{array} \right.
\end{array}
\label{Eqn:DiscrMain}
\end{equation}
The forward-difference stencil for this method is presented in Fig. \ref {Fig:Scheme}.

In the unbounded situation (that is, when we consider a semi-infinite system of coupled Josephson junctions), we choose a large system of $N$ Josephson junctions following (\ref {Eqn:Main2}), in which the external damping coefficient includes both the effect of external damping and a simulation of an absorbing boundary slowly increasing in magnitude on the last $N - N _0$ oscillators. Explicitly, we let $\gamma$ be the sum of the actual value of the external damping coefficient and the function 
\begin{equation}
\gamma ^{\prime} (n) = 0.5 \left[ 1 + \tanh \left( \displaystyle {\frac {2 n - N _0 + N} {6}} \right) \right],
\end{equation}
where $N _0 = 50$ and $N \geq 200$ for our computations.

It is important to posses discrete schemes to approximate the local energy density and the total energy of the problem under study. We associate the following expression to finite-difference scheme (\ref {Eqn:DiscrMain}), in order to approximate the value of the Hamiltonian of the $n$th site at the $k$th time:
\begin{equation}
\begin{array}{rcl}
H _n ^k & = & \displaystyle {\frac {1} {2} \left( \frac {u _n ^{k + 1} - u _n ^k} {\Delta t} \right) ^2 + \frac {c ^2} {8} \left[ \left( u _{n + 1} ^{k + 1} - u _n ^{k + 1} \right) ^2 \right.} \\
  & & \displaystyle {\left. + \left( u _{n - 1} ^{k + 1} - u _n ^{k + 1} \right) ^2 + \left( u _{n + 1} ^k - u _n ^k \right) ^2 + \left( u _{n - 1} ^k - u _n ^k \right) ^2 \right]} \\
 & & \quad \displaystyle{ + \frac {\mathfrak {m} ^2} {2} \frac {(u _n ^{k + 1}) ^2 + (u _n ^k) ^2} {2} + \frac {V (u _n ^{k + 1}) + V (u _n ^k)} {2} - J \frac {u _n ^{k + 1} + u _n ^k}{2}}.
\end{array} \label{Eqn:Hamiltonian0}
\end{equation}
Moreover, the total energy of the system is estimated by
\begin{equation}
E _k = \displaystyle {\sum _{n = 1} ^ N H _n ^k + \frac {c ^2} {4} \left[ \frac {\left( u _1 ^{k + 1} - u _0 ^{k + 1} \right) ^2 + \left( u _1 ^k - u _0 ^k \right) ^2} {2} \right].}
\label{Eqn:DEner0}
\end{equation}

\subsection{A second finite-difference scheme}

Let $\delta ^2 u _n ^k = u _n ^{k + 1} + 2 u _n ^k + u _n ^{k - 1}$. A second, finite-difference scheme to approximate solutions to problem (\ref {Eqn:Main2}) is presented next. Its forward-difference stencil is depicted in Fig. \ref {Fig:Scheme2}:
\begin{equation}
\begin{array}{c}
\begin{array}{rl}\left.
\begin{array}{rcl}
\displaystyle {\frac {\delta ^2 _t u _n ^k} {(\Delta t) ^2} - \left( \frac {c ^2} {4} {\delta ^2} + \frac {\beta} {2 \Delta t} \delta _t \right) \delta ^2 _x u _n ^k + \frac {\gamma} {2 \Delta t} \delta _t u _n ^k +} \quad & & \\
\displaystyle {\frac {\mathfrak {m} ^2} {2} \bar {\delta} u _n ^k + \frac {V (u _n ^{k + 1}) - V (u _n ^{k - 1})} {u _n ^{k + 1} - u _n ^{k - 1}}} - J & = & 0,
\end{array}\right\} & 
\begin{array}{l}
1 \leq n \leq N, \\
1 \leq k \leq M,
\end{array}
\end{array} \\
\textrm {s.\ t.} \left\{ \begin{array}{ll}
u _n ^0 = 0, \quad u _n ^1 = 0, & 1 \leq n \leq N, \\
c ^2 \bar {\delta} (u _0 ^k - u _1 ^k) + \beta \delta _t (u _0 ^k - u _1 ^k) / \Delta t = 2 \phi _k, & 1 \leq k \leq M, \\
u _{N + 1} ^k - u _N ^k = 0, & 1 \leq k \leq M.
\end{array} \right.
\end{array}
\label{Eqn:DiscrMain2}
\end{equation}

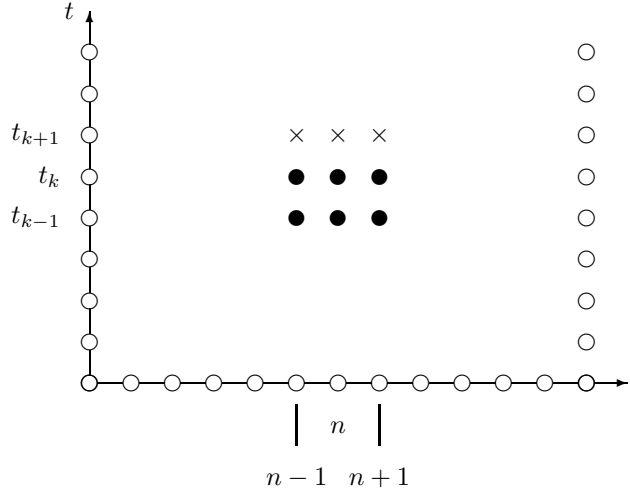
\begin{figure}
\begin{center}
\setlength{\unitlength}{.55mm}
\begin{picture}(140,100)
\put(10,12){\line(0,1){6}} %
\put(10,22){\line(0,1){6}} %
\put(10,32){\line(0,1){6}} %
\put(10,42){\line(0,1){6}} %
\put(10,52){\line(0,1){6}} %
\put(10,62){\line(0,1){6}} %
\put(10,72){\line(0,1){6}} %
\put(10,82){\line(0,1){6}} %
\put(10,92){\vector(0,1){8}} %
\put(5,100){\makebox(0,0)[5]{$t$}} %
\put(12,10){\line(1,0){6}} %
\put(22,10){\line(1,0){6}} %
\put(32,10){\line(1,0){6}} %
\put(42,10){\line(1,0){6}} %
\put(52,10){\line(1,0){6}} %
\put(62,10){\line(1,0){6}} %
\put(72,10){\line(1,0){6}} %
\put(82,10){\line(1,0){6}} %
\put(92,10){\line(1,0){6}} %
\put(102,10){\line(1,0){6}} %
\put(112,10){\line(1,0){6}} %
\put(122,10){\line(1,0){6}} %
\put(132,10){\vector(1,0){8}} %
\put(10,10){\circle{4}} %
\put(10,20){\circle{4}} %
\put(10,30){\circle{4}} %
\put(10,40){\circle{4}} %
\put(10,50){\circle{4}} %
\put(10,60){\circle{4}} %
\put(10,70){\circle{4}} %
\put(10,80){\circle{4}} %
\put(10,90){\circle{4}} %
\put(130,10){\circle{4}} %
\put(130,20){\circle{4}} %
\put(130,30){\circle{4}} %
\put(130,40){\circle{4}} %
\put(130,50){\circle{4}} %
\put(130,60){\circle{4}} %
\put(130,70){\circle{4}} %
\put(130,80){\circle{4}} %
\put(130,90){\circle{4}} %
\put(10,10){\circle{4}} %
\put(20,10){\circle{4}} %
\put(30,10){\circle{4}} %
\put(40,10){\circle{4}} %
\put(50,10){\circle{4}} %
\put(60,10){\circle{4}} %
\put(70,10){\circle{4}} %
\put(80,10){\circle{4}} %
\put(90,10){\circle{4}} %
\put(100,10){\circle{4}} %
\put(110,10){\circle{4}} %
\put(120,10){\circle{4}} %
\put(130,10){\circle{4}} %
\put(70,-2){\makebox(0,0)[b]{$n$}} %
\put(60,-15){\makebox(0,0)[b]{$n - 1$}} %
\put(60,-5){\line(0,1){10}} %
\put(80,-5){\line(0,1){10}} %
\put(80,-15){\makebox(0,0)[b]{$n + 1$}} %
\put(3,60){\makebox(0,0)[r]{$t _k$}} %
\put(3,70){\makebox(0,0)[r]{$t _{k + 1}$}} %
\put(3,50){\makebox(0,0)[r]{$t _{k - 1}$}} %
\put(70,60){\circle*{4}} %
\put(80,60){\circle*{4}} %
\put(60,60){\circle*{4}} %
\put(70,50){\circle*{4}} %
\put(60,50){\circle*{4}} %
\put(80,50){\circle*{4}} %
\put(60,70){\makebox(0,0){$\times$}} %
\put(70,70){\makebox(0,0){$\times$}} %
\put(80,70){\makebox(0,0){$\times$}} %
\end{picture} 
\end{center}\ \smallskip
\caption{Forward-difference stencil for the approximation to the $n$th differential equation in (\ref {Eqn:Main2}) at time $t _k$, using scheme (\ref {Eqn:DiscrMain2}). The black circles represent known approximations to the actual solutions at times $t _{k - 1}$ and $t _k$ in scheme (\ref {Eqn:DiscrMain2}), and the crosses denote the unknown approximations at time $t _{k + 1}$.  \label{Fig:Scheme2}}
\end{figure}

It is worth mentioning that finite-difference schemes (\ref {Eqn:DiscrMain}) and (\ref {Eqn:DiscrMain2}) are modified versions of the one presented in \cite {Macias-Supra}, which in turn is a spatially discretized and adapted version of a computational method to approximate radially symmetric solutions of initial-value problems involving a general class of continuous, three-dimensional Klein-Gordon equations \cite {Macias-Puri}. 

For approximations obtained via (\ref {Eqn:DiscrMain2}), the Hamiltonian of the $n$th site of the system will be computed using the expression
\begin{equation}
\begin{array}{rcl}
H _n ^k & = & \displaystyle {\frac {1} {2} \left( \frac {u _n ^{k + 1} - u _n ^k} {\Delta t} \right) ^2 + \frac {c ^2} {4} \left[ \left( \frac {u _{n + 1} ^{k + 1} + u _{n + 1} ^k} {2} - \frac {u _n ^{k + 1} + u _n ^k} {2} \right) ^2 \right.} \\
  & & \displaystyle {\left. + \left( \frac {u _{n - 1} ^{k + 1} + u _{n - 1} ^k} {2} - \frac {u _n ^{k + 1} + u _n ^k} {2} \right) ^2 \right]} \\
 & & \quad \displaystyle {+ \frac {\mathfrak {m} ^2} {2} \frac {(u _n ^{k + 1}) ^2 + (u _n ^k) ^2} {2} + \frac {V (u _n ^{k + 1}) + V (u _n ^k)} {2} - J \frac {u _n ^{k + 1} + u _n ^k}{2}}.
\end{array} \label{Eqn:Hamiltonian}
\end{equation}
With this notation at hand, the total energy of the system will be approximated via
\begin{equation}
E _k = \displaystyle {\sum _{n = 1} ^ N H _n ^k + \frac {c ^2} {4} \left(\frac {u _1 ^{k + 1} + u _1 ^k} {2} - \frac {u _0 ^{k + 1} + u _0 ^k} {2} \right) ^2.}
\label{Eqn:DEner}
\end{equation}

It is important to mention that there are several computational reasons to prefer an implicit scheme over an explicit one, in order to approximate solutions to (\ref {Eqn:Main2}). First of all, it has been noted that some explicit schemes used to approximate solutions to modified versions of our problem have proved to be highly unstable \cite {StraussVazquez}; on the other hand, the use of an implicit scheme seems inevitable in order to approximate the term $D _t \Delta _x ^2 u _n$ with a consistency of order at least $\mathcal {O} (\Delta t) ^2$. 

Also, we must mention that the proposed numerical methods are consistent with respect to the problem under study, nonlinear, and require an application of Newton's method for systems of nonlinear equations, together with an application of Crout's reduction technique for tridiagonal systems \cite {Burden}. 

Finally, it is worthwhile mentioning here that the stopping criterion used for Newton's method in our computations was 
\begin{equation}
\max \{ | u _n ^{k + 1} - u _n ^k | : 1 \leq n \leq N \} < 1 \times 10 ^{- 5}.
\end{equation}

\section{Numerical properties}


In this section, we prove that finite-difference schemes (\ref {Eqn:DiscrMain}) and (\ref {Eqn:DiscrMain2}) posses properties of consistency that make them useful methods in the study of the process of nonlinear supratransmission in chains of coupled Josephson junctions. With this aim in mind, we present here the most important results concerning the energy properties of our methods.

\begin{lemma} 
The discrete rate of change of energy of a sequence $(u _n ^k)$ satisfying finite-difference scheme (\ref {Eqn:DiscrMain}) is given by
\begin{equation}
\begin{array}{rcl}
\displaystyle {\frac {E _k - E _{k - 1}} {\Delta t}} & = & \displaystyle {- \beta \left\{ \sum _{n = 1} ^N \left( \frac {\delta _t u _n ^k - \delta _t u _{n - 1} ^k} {2 \Delta t} \right) ^2 + \left( \frac {\delta _t u _1 ^k - \delta _t u _0 ^k} {2 \Delta t} \right) \frac {\delta _t u_0 ^k} {2 \Delta t} \right\}} \\
 & & \quad \displaystyle {- \gamma \sum _{n = 1} ^N \left( \frac {\delta _t u _n ^k} {2 \Delta t} \right) ^2 + c ^2 \left( \frac {\bar {\delta} u _0 ^k - \bar {\delta} u _1 ^k} {2} \right) \frac {\delta _t u _0 ^k} {2 \Delta t}}.
\end{array}
\end{equation}\label{LemmaChi0}
\end{lemma}

\begin{proof}
We provide a sketch of the proof in view that the task is tedious but, after all, a mere algebraic work. Term by term, it is straight-forward to check that 
\begin{eqnarray}
\sum _{n = 1} ^N \left[ \left( \frac {u _n ^{k + 1} - u _n ^k} {\Delta t} \right) ^2 - \left( \frac {u _n ^k - u _n ^{k - 1}} {\Delta t} \right) ^2 \right] & = & \sum _{n = 1} ^N \left( \frac {\delta _t ^2 u _n ^k} {\Delta t ^2} \right) \delta _t u _n ^k, \\
\sum _{n = 1} ^N  \left[ \frac {(u _n ^{k + 1}) ^2 + (u _n ^k) ^2} {2} - \frac {(u _n ^k) ^2 + (u _n ^{k - 1}) ^2} {2} \right] & = & \sum _{n = 1} ^N \left( \frac {u _n ^{k + 1} + u _n ^{k - 1}} {2} \right) \delta _t u _n ^k, \\
\sum _{n = 1} ^N \left\{ [V (u _n ^{k + 1}) + V (u _n ^k)] - [V (u _n ^k) - V (u _n ^{k - 1})] \right\} & = & \sum _{n = 1} ^N \left[ \frac {V (u _n ^{k + 1}) - V (u _n ^{k - 1})} {u _n ^{k + 1} - u _n ^{k - 1}} \right] \delta _t u _n ^k.
\end{eqnarray}
For the sake of convenience, we let 
\begin{equation}
B _k = \displaystyle {\frac {c ^2} {8} \sum _{n = 1} ^N \left[ \left( u _{n + 1} ^{k + 1} - u _n ^{k + 1} \right) ^2 + \left( u _{n - 1} ^{k + 1} - u _n ^{k + 1} \right) ^2 + \left( u _{n + 1} ^k - u _n ^k \right) ^2 + \left( u _{n - 1} ^k - u _n ^k \right) ^2 \right].}
\end{equation}
Using finite-difference scheme (\ref {Eqn:DiscrMain}) and the discrete, boundary conditions imposed on the problem, it is straight-forward to check that 
\begin{equation}
\begin{array}{rcl}
B _k - B _{k - 1} & = & \displaystyle {\frac {c ^2} {8} \sum _{n = 1} ^N \left[ - 2 \left( \bar {\delta} \delta _x ^2 u _n ^k \right) \left( \delta _t u _n ^k \right) \right.} \\
  & & \displaystyle {\left. + \left( \bar {\delta} u _{n + 1} ^k - \bar {\delta} u _n ^k \right) \left( \delta _t u _{n + 1} ^k \right) - \left( \bar {\delta} u _n ^k - \bar {\delta} u _{n - 1} ^k \right) \left( \delta _t u _n ^k \right) \right.} \\
 & & \displaystyle {\quad \left. + \left( \bar {\delta} u _{n + 1} ^k - \bar {\delta} u _n ^k \right) \left( \delta _t u _n ^k \right) - \left( \bar {\delta} u _n ^k - \bar {\delta} u _{n - 1} ^k \right) \left( \delta _t u _{n - 1} ^k \right) \right].}
\end{array}
\end{equation}
Notice that the last four terms under the summation symbol form two telescoping series, which may be readily simplified. Next, the difference $E _k - E _{k - 1}$ is computed. After some simplification, we obtain that
\begin{equation}
\frac {E _k - E _{k - 1}} {\Delta t} = \beta \sum _{n = 1} ^N \left( \frac {\delta _t \delta _x ^2 u _n ^k} {2 \Delta t} \right) \left( \frac {\delta _t u _n ^k} {2 \Delta t} \right) - \gamma \sum _{n = 1} ^N \left( \frac {\delta _t u _n ^k} {2 \Delta t} \right) ^2 + c ^2 \left( \frac {\bar {\delta} u _0 ^k - \bar {\delta} u _1 ^k} {2} \right) \frac {\delta _t u _0 ^k} {2 \Delta t}.
\end{equation}
The result follows now after simplifying the term multiplied by $\beta$, followed by an application of the discrete version of Green's First Identity \cite {Macias-Supra} to the sequence $(u _n ^{k + 1} - u _n ^{k - 1}) _{n = 0} ^{N + 1}$.
\end{proof}

\begin{lemma}
The discrete rate of change of energy of a sequence $(u _n ^k)$ satisfying finite-difference scheme (\ref {Eqn:DiscrMain2}) is given by
\begin{equation}
\begin{array}{rcl}
\displaystyle {\frac {E _k - E _{k - 1}} {\Delta t}} & = & \displaystyle {- \beta \left\{ \sum _{n = 1} ^N \left( \frac {\delta _t u _n ^k - \delta _t u _{n - 1} ^k} {2 \Delta t} \right) ^2 + \left( \frac {\delta _t u _1 ^k - \delta _t u _0 ^k} {2 \Delta t} \right) \frac {\delta _t u_0 ^k} {2 \Delta t} \right\}} \\
 & & \quad \displaystyle {- \gamma \sum _{n = 1} ^N \left( \frac {\delta _t u _n ^k} {2 \Delta t} \right) ^2 + c ^2 \left[ \left(\frac {\delta ^2 u _0 ^k - \delta ^2 u _1 ^k} {4}\right) \frac {\delta _t u _0 ^k} {2 \Delta t} \right]}.
\end{array}
\end{equation}\label{LemmaChi}
\end{lemma}

\begin{proof}
Here, we employ some of the identities used in the proof of Proposition \ref {LemmaChi0}. Notice first of all that if we define $C _k$ by the formula
\begin{equation}
C _k = \frac {c ^2} {4} \sum _{n = 1} ^N \left[ \left( \frac {u _{n + 1} ^{k + 1} + u _{n + 1} ^k} {2} - \frac {u _n ^{k + 1} + u _n ^k} {2} \right) ^2 + \left( \frac {u _{n - 1} ^{k + 1} + u _{n - 1} ^k} {2} - \frac {u _n ^{k + 1} + u _n ^k} {2} \right) ^2 \right],
\end{equation}
then
\begin{equation}
\begin{array}{rcl}
C _k - C _{k - 1} & = & \displaystyle {\frac {c ^2} {4} \left\{ \left( \frac {\delta ^2 u _0 ^k - \delta ^2 u _1 ^k} {4} \right) \left( \delta _t u _0 ^k + \delta _t u _1 ^k \right) - \sum _{n = 1} ^N \left( \frac {\delta ^2 \delta _x ^2 u _n ^k } {2} \right) \delta _t u _n ^k \right\} }.
\end{array}
\end{equation}
We mimic here the proof of Proposition \ref {LemmaChi0}. As a consequence of the fact that the sequence $(u _n ^k)$ satisfies (\ref {Eqn:DiscrMain2}) and after an application of the discrete version of Green's First Identity \cite{Macias-Supra}, it is readily checked that the discrete rate of change of the energy between times $k$ and $k - 1$ is given by
\begin{equation}
\begin{array}{rcl}
\displaystyle {\frac {E _k - E _{k - 1}} {\Delta t}} & = & \displaystyle { - \beta \left\{ \sum _{n = 1} ^N \left( \frac {\delta _t u _n ^k - \delta _t u _{n - 1} ^k} {2 \Delta t} \right) ^2 + \left( \frac {\delta _t u _1 ^k - \delta _t u _0 ^k} {2 \Delta t} \right) \frac {\delta _t u _0 ^k} {2 \Delta t} \right\}} \\
  & & \displaystyle {- \gamma \sum _{n = 1} ^N \left( \frac {\delta _t u _n ^k} {2 \Delta t} \right) ^2 + \frac {c ^2} {2} \left( \frac {\delta ^2 u _0 ^k - \delta ^2 u _1 ^k} {4} \right) \left( \frac {\delta _t u _1 ^k + \delta _t u _0 ^k} {2 \Delta t} \right)} \\
  & & \displaystyle {+ \frac {c ^2} {2} \left( \frac {\delta ^2 u _0 ^k - \delta ^2 u _1 ^k} {4} \right) \left( \frac {\delta _t u _0 ^k - \delta _t u _1 ^k} {2 \Delta t} \right)},
\end{array}
\end{equation}
whence the result follows after an easy simplification.
\end{proof}


The methods described by finite-difference schemes (\ref {Eqn:DiscrMain}) and (\ref {Eqn:DiscrMain2}) are clearly consistent of order $O(\Delta t)^2$. Moreover, in the following, we give some necessary and sufficient conditions in order for these methods to be stable. 
Notice that when $\mathfrak {m}$ is a real number, problem (3) represents a discrete, nonlinear modification of the classical linear Klein-Gordon equation.

\begin{lemma}
Consider finite-difference scheme (\ref {Eqn:DiscrMain2}) with $V ^\prime = 0$, $J = 0$ and $\mathfrak {m} ^2$  a nonnegative constant. Then, $\Delta t \leq \frac {1} {c}$ is a necessary condition for method (\ref {Eqn:DiscrMain2}) to be linearly stable.
\label{Lemma:Silvia1}
\end{lemma}

\begin{figure}
\centerline{\includegraphics[width=0.8\textwidth]{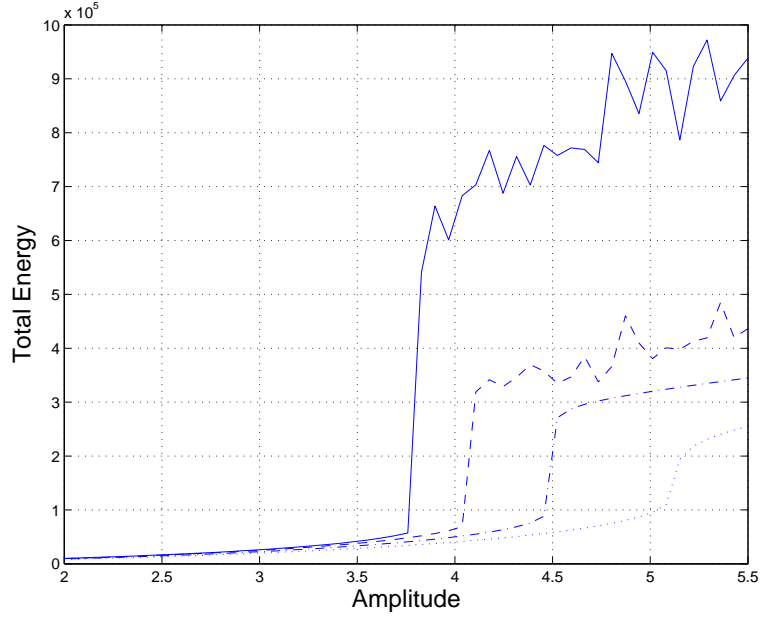}}
\caption{Graph of total energy versus driving amplitude for a system (\ref {Eqn:Main2}) subject to a harmonic frequency of $0.8$, with coupling coefficient equal to $5$, and a time period of $10000$. Different values of external damping were chosen: $\gamma = 0$ (solid), $0.1$ (dashed), $0.2$ (dash-dotted), and $0.3$ (dotted).\label{Fig1}}
\end{figure}

\begin{proof}
We apply the von Neumann stability analysis to method (\ref {Eqn:DiscrMain2}) and we obtain the following augmentation matrix:
\begin{equation}
A (\xi) = \left( %
\begin{array}{cc}
\frac {\hat {f} ( \xi )} {\hat {g} ( \xi )} & -
\frac {\hat {h} ( \xi )} {\hat {g} ( \xi )} \\
1 & 0
\end{array}
\right), \label{Eqn:Matrix2}
\end{equation}
where
\begin{equation}
\begin{array}{rcl}
\hat {f} (\xi) & = & \displaystyle {2 - 2 c ^2 (\Delta t) ^2
\sin ^2 \left( \frac {\xi} {2} \right)}, \\
\hat {g} (\xi) & = & \displaystyle {1 + \left[ c ^2 (\Delta t) ^2
+ 2 \beta \Delta t \right] \sin ^2 \left(\frac {\xi} {2}\right) +
\frac {\mathfrak {m} ^2 (\Delta t) ^2 + \gamma \Delta t} {2}, \quad \textrm {and}} \\
\hat {h} (\xi) & = & \displaystyle {1 + \left[ c ^2 (\Delta t) ^2 -
2 \beta \Delta t \right] \sin ^2 \left(\frac {\xi} {2}\right) +
\frac {\mathfrak {m} ^2 (\Delta t) ^2 - \gamma \Delta t} {2}.}
\end{array}
\end{equation}
Since all norms in a finite-dimensional space are equivalent and the spectral radius $\rho (A)$ of a square matrix $A$ is the greatest lower bound for the set of matrix-induces norms of $A$, then a necessary condition for linear stability is that $\rho (A)\leq 1$. The
eigenvalues of $A (\xi)$ are:
\begin{equation}
 \lambda _\pm(\xi)  = \displaystyle{\frac {\hat {f} (\xi) \pm
\sqrt {\hat {f} ^2 (\xi) - 4 \hat {g} (\xi) \hat {h} (\xi)}} {2 \hat
{g} (\xi)}}. \label{eigenvalues}
\end{equation}
In order to show that the spectral radius of $A( \xi )$ is less than or equal to 1, we will prove that $\lambda
_\pm(\xi)\leq 1$ for all  $\xi\in [0,\pi]$ by considering two cases:
\begin{enumerate}
\item[(1)] If $\lambda _\pm (\xi) \in \mathbb {C}$ then $\hat {f} ^2 (\xi) - 4 \hat {g} (\xi) \hat {h} (\xi) \leq 0$ or,
equivalently, that $4 \hat {g} (\xi) \hat {h} (\xi) \geq \hat {f} ^2 (\xi)$, whence it follows that $\hat {h} (\xi)$ is nonnegative. Moreover,
\begin{equation}
| \lambda _\pm (\xi) | = \sqrt {\frac {\hat {f} ^2 (\xi) + (4 \hat {g} (\xi) \hat {h} (\xi) - \hat {f} ^2 (\xi))} {4 \hat {g} ^2 (\xi)}} = \sqrt {\frac {\hat {h} (\xi)} {\hat {g} (\xi)}} \leq 1.
\label{Eqn:complej2}
\end{equation}
\item[(2)] If $\lambda_\pm (\xi) \in \mathbb{R}$ and $\Delta t \leq \frac {1} {c}$, then we observe first of all that $ \hat {f} (\xi)$ is nonnegative with $\hat {f} ^2 (\xi) \leq 4$. This case leads to the chain of inequalities $1 \geq \hat {h} (\xi) \geq 2 - \hat {g} (\xi) \geq \hat {f} (\xi) - \hat {g} (\xi)$. Clearly, $2 \hat {g} (\xi) \geq \hat {f} (\xi)$, so that
\begin{equation}
| \lambda _\pm (\xi) | \leq \frac {\hat {f} (\xi) + \sqrt {\hat {f} ^2 (\xi) - 4 \hat {g} (\xi) \hat {h} (\xi)}} {2 \hat {g} (\xi)} \leq \frac {\hat {f} (\xi) + \sqrt {(2 \hat {g} (\xi) - \hat {f} (\xi)) ^2}} {2 \hat {g} (\xi)} = 1.\label{Eqn:real2}
\end{equation}
\end{enumerate}
\end{proof}

\begin{figure}
\centerline{\includegraphics[width=0.8\textwidth]{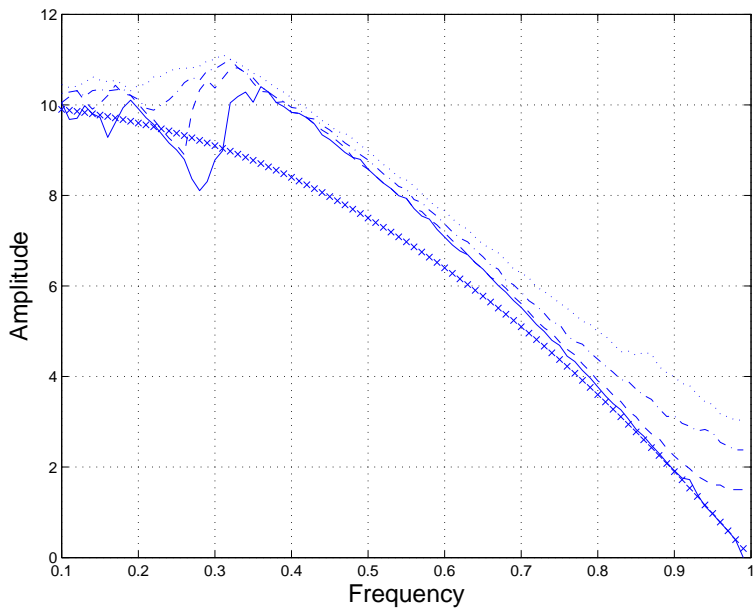}}
\caption{Diagram of smallest driving amplitude at which supratransmission occurs versus driving frequency for a system (\ref {Eqn:Main2}) subject to a harmonic frequency with coupling coefficient equal to $5$. Different values of external damping were chosen: $\gamma = 0$ (solid), $0.1$ (dashed), $0.2$ (dash-dotted), and $0.3$ (dotted).\label{Fig2}}
\end{figure}

\begin{lemma}
Consider finite-difference scheme (\ref {Eqn:DiscrMain2}) with $V ^\prime = 0$, $J = 0$ and $\mathfrak {m} ^2$ a nonnegative constant. This scheme is linearly stable in the infinite norm if $\frac{2}{\gamma} < \Delta t < \frac {\sqrt {2}} {c}$. \label{Lemma:Silvia2}
\end{lemma}

\begin{proof}
A sufficient condition for scheme (\ref{Eqn:DiscrMain2}) to be linearly stable is that $||A(\xi)||\leq 1$ for all $\xi \in [0 , \pi]$, for some norm $\Vert \cdot \Vert$; in our case, we will consider the infinite norm $\Vert \cdot \Vert _\infty$ applied to the augmentation matrix (\ref{Eqn:Matrix2}). Using the same notation as in the previous result, observe that the inequality $\Vert A (\xi) \Vert_\infty \leq 1$ is satisfied in case that 
\begin{equation}
\displaystyle\Bigl|\frac{2}{\hat {g} (\xi)}\Bigr|+\Bigl|-\frac{\hat
{h} (\xi)}{\hat {g} (\xi)}\Bigr| \leq 1 .\label{Eqn:cond1}
\end{equation}
In turn, condition (\ref{Eqn:cond1}) holds if $2\leq \hat {g} (\xi)+\hat
{h} (\xi)$ and $2\leq \hat {g} (\xi)-\hat {h} (\xi)$. But the former inequality is always true, while the latter is satisfied if and only if
\begin{equation}
1 \leq 2 \beta \Delta t  \sin ^2 \left(\frac {\xi} {2}\right) + \frac {\gamma \Delta t} {2}, \quad \forall \xi \in [0 , \pi]. \label{Eqn:cond2}
\end{equation}
\end{proof}

\begin{corollary}
The finite-difference scheme (\ref{Eqn:DiscrMain}) is linearly stable if and only if condition $\frac{2}{\gamma}\leq \Delta t \leq \frac{1}{c}$  holds.\label{Lemma:Silvia3}
\end{corollary}

\begin{proof}
We apply again the von Neumann stability analysis for the scheme (\ref{Eqn:DiscrMain}). The augmentation matrix for this scheme is the same as (\ref {Eqn:Matrix2}) when we set $\hat {f} (\xi) = 2$ for all $\xi \in [0 , \pi]$ and the proof of Propositions \ref{Lemma:Silvia1} and \ref{Lemma:Silvia2} are valid for this case.
\end{proof}

\begin{figure}
\centerline{\includegraphics[width=0.8\textwidth]{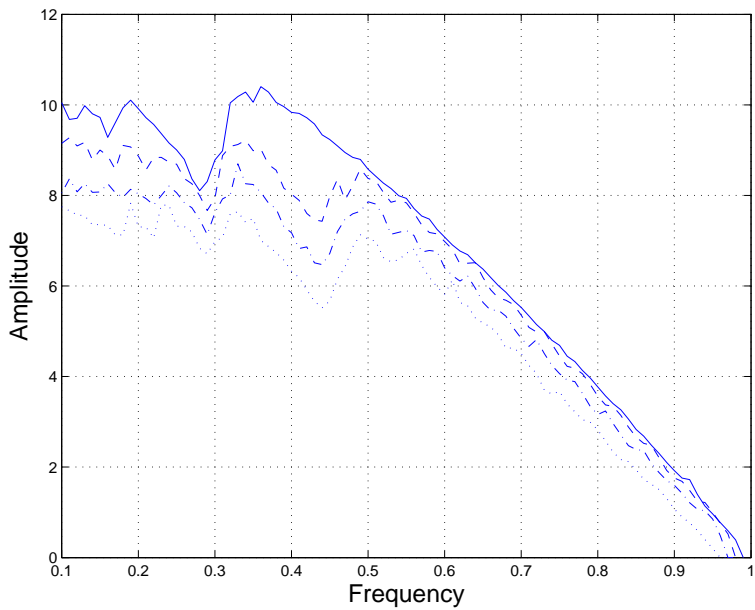}}
\caption{Diagram of smallest driving amplitude at which supratransmission occurs versus driving frequency for a system (\ref {Eqn:Main2}) subject to a harmonic frequency with coupling coefficient equal to $5$. Different values of the Josephson current were chosen: $J = 0$ (solid), $0.1$ (dashed), $0.2$ (dash-dotted), and $0.3$ (dotted).\label{Fig4}}
\end{figure}

\section{Application\label {Sec5}}

Before we provide an application of the numerical techniques presented in this work, we must declare that the results of this section were obtained by means of (\ref {Eqn:DiscrMain}) and (\ref {Eqn:DiscrMain2}). Simulations in both the domain of solutions and the energy domain have shown that the methods agree excellently. Moreover, it is important to mention that we have performed several numerical simulations (not shown here) with both methods, for several nonnegative values of $\beta$, $\gamma$ and $\mathfrak {m} ^2$, several values of the positive constant $c ^2$ and both $V ^\prime$ and $J$ equal to zero, observing stability of the methods when the corresponding stability conditions are satisfied. 

The validity of the implementation of our methods is checked next against the prediction of the nonlinear supratransmission threshold for a semi-infinite, linear array of Josephson junctions in parallel and coupled through superconducting wires, for which the mathematical model is given by (\ref {Eqn:Main2}). Throughout, we employ a coupling coefficient equal to $5$, and follow the systematic procedure used in \cite {Macias-Supra}, namely:
\begin{enumerate}
\item For a fixed frequency $\Omega$ in the forbidden band-gap, we examine the behavior of solutions of our problem for different driving amplitudes, using the finite-difference schemes proposed in this article. According to the theory, there exists a critical amplitude above which the qualitative nature of the solutions drastically changes.

\item This drastic change in the behavior of the solutions is checked next in the energy domain, by making use of the energy schemes associated to each finite-difference scheme. A sudden increase in the total energy injected into the medium by the driving boundary must be observed above the predicted supratransmission threshold.

\item Finally, this procedure is applied to a regular partition $\{ \Omega _i \} _{i = 0} ^l$ of the frequency interval $[0 , \sqrt {\mathfrak {m} ^2 + 1}]$, obtaining thus a second array $\{ A _i \} _{i = 0} ^l$ containing, for each driving frequency $\Omega _i$, the corresponding critical amplitude $A _i$ at which supratransmission starts. In this way, a graph of occurrence of critical amplitude versus driving frequency may be constructed.

\item This procedure may be repeated for different choices of the parameters of the model studied.
\end{enumerate}

\begin{figure}[tbc]
\centerline{\includegraphics[width=0.8\textwidth]{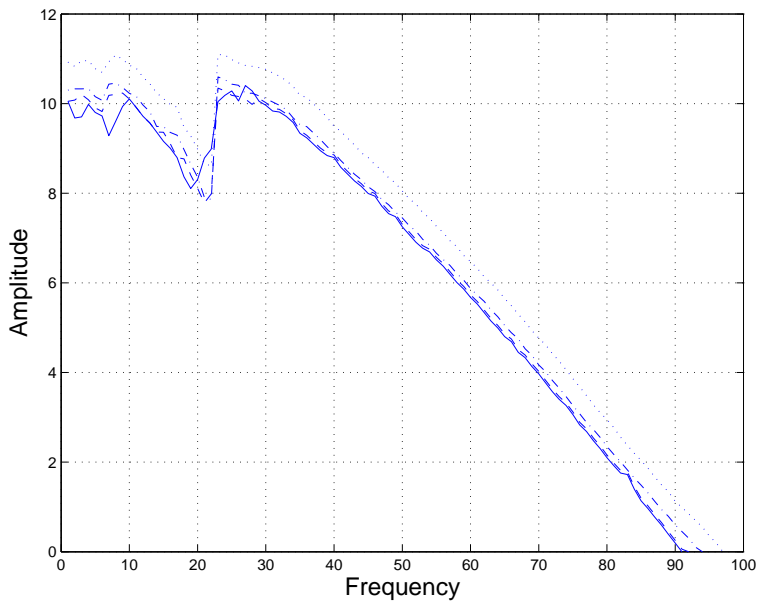}}
\caption{Diagram of smallest driving amplitude at which supratransmission occurs versus driving frequency for a system (\ref {Eqn:Main2}) subject to a harmonic frequency with coupling coefficient equal to $5$. Different values of internal damping were chosen: $\beta = 0$ (solid), $0.1$ (dashed), $0.3$ (dash-dotted), and $0.6$ (dotted).\label{Fig6}}
\end{figure}

Let us fix a driving frequency $\Omega = 0.8$ in the forbidden band-gap of system (\ref {Eqn:Main2}), let $c = 5$, and set all the other scalar parameters of the model equal to zero. Select a fixed period of time of $10000$, during which the system will be perturbed by the harmonic driving $\phi (t) = A \sin (\Omega t)$, where the value of $A$ will be chosen inside an interval that contains the predicted threshold. According to \cite {Chevrieux2}, the critical value of the continuous-limit case is provided by the expression 
\begin{equation}
A _s = 2 c (1 - \Omega ^2), 
\end{equation}
which yields a critical amplitude equal to $3.6$ in our case. In these circumstances, Fig. \ref {Fig1} presents the graph of the total energy of the system for values of the driving amplitude between $2$ and $5.5$, in the form of a solid line. A drastic increase in the total energy is observed around an amplitude of $3.75$, in good agreement with the continuous-limit case.

\begin{figure}
\centerline{\includegraphics[width=0.8\textwidth]{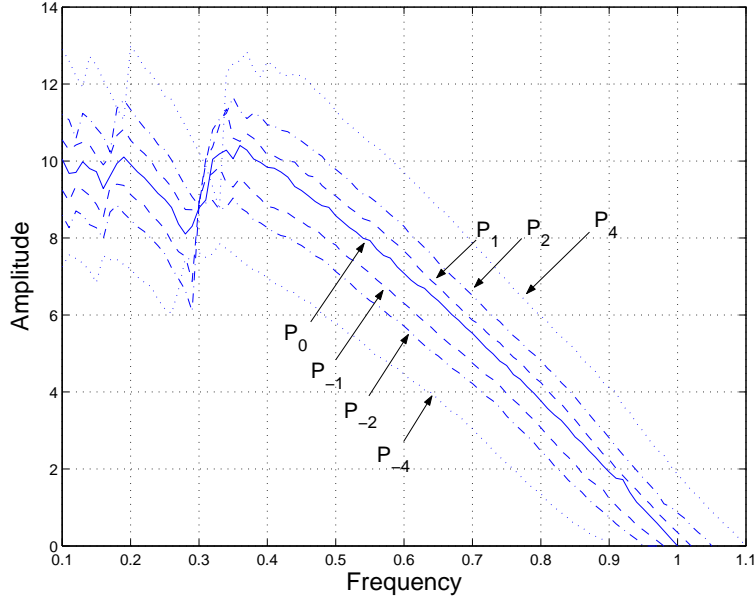}}
\caption{Diagram of smallest driving amplitude at which supratransmission occurs versus driving frequency for a system (\ref {Eqn:Main2}) subject to a harmonic frequency with coupling coefficient equal to $5$. For every $\ell = 0 , \pm 1 , \pm 2 , \pm 4$, the plot P$_\ell$ corresponds to a system with mass $\mathfrak {m}$ satisfying $\sqrt {\mathfrak {m} _\ell ^2 + 1} = 1 + \frac {\ell} {40}$.\label{Fig5}}
\end{figure}

Let us denote by $E _1 = E _1 (A)$ the total energy of the system described in the previous paragraph associated with the amplitude $A \in [2 , 5.5]$ at time $10000$, obtained using finite-difference scheme (\ref {Eqn:DiscrMain}). Similarly, let $E _2 = E _2 (A)$ be the corresponding energy obtained using scheme (\ref {Eqn:DiscrMain2}). It is important to mention that our computations show that
\begin{equation}
\max \{ |E _1 (A) - E _2 (A)| : A \in [2 , 5.5] \} < 3 \times 10 ^{- 10}.
\end{equation}
Moreover, for amplitudes below the critical threshold,
\begin{equation}
\max \{ |E _1 (A) - E _2 (A)| : A \in [2 , 3.5] \} < 8 \times 10 ^{- 16}.
\end{equation}

Next we wish to check the effect of external damping on the occurrence of the supratransmission value. To this effect, we fix three different values of external damping: $\gamma = 0.1$, $0.2$ and $0.3$; the results are displayed in Fig. \ref {Fig1}, in which the corresponding graphs are presented as dashed, dashed-dotted and dotted lines, respectively. It is worth noticing that the presence of external damping produces a right shift in the occurrence of the supratransmission value and, at the same time, a decrease in the total energy of the system. Similar remarks have been made for the same problem with Dirichlet boundary data \cite {Macias-Supra}.

In a next sage, we wish to compute diagrams for the occurrence of supratransmission in the presence of all parameters. So, we compute graphs of driving amplitude at which supratransmission first starts versus driving frequency for system (\ref {Eqn:Main2}). Under the same computational setting as before, Figs. \ref {Fig2} and \ref {Fig4} present such diagrams for different values of the external damping coefficient and the generalized Josephson current, respectively. A quick comparison with similar results obtained using a different finite-difference scheme \cite {Macias-Signals3} leads us to certify the validity of our method. Additionally, we provide diagrams which evidence the physical implications of including a relativistic mass and internal damping (see Figs. \ref {Fig5} for mass, and \ref {Fig6} for internal damping), the results being in qualitative agreement with the Dirichlet case, in the following senses:

\begin{itemize}
\item The phenomenon of harmonic phonon quenching still appears in the presence of external and internal damping.
\item The discrepancy region due to phonon quenching is shortened as the external damping coefficient is increased.
\item The threshold value at which supratransmission first occurs, for fixed frequencies outside the discrepancy region, increases as the damping coefficient (internal or external) increases.
\item The diagram, for any fixed value of $\beta \geq 0$, is approximately equal to the corresponding diagram for the undamped system, shifted $\beta$ horizontal units to the right. 
\item A horizontal shift of $\sqrt {1 + \mathfrak {m} ^2} - 1$ units in the diagram of a sine-Gordon system of mass $\mathfrak {m}$ with respect to the corresponding massless system, is observed for small masses and frequencies outside the discrepancy region.
\end{itemize}

\section{Conclusions\label{Sec:Concl}}

In this article, we have presented two numerical methods to approximate solutions of a system of differential equations, which models discrete arrays consisting of Josephson junctions coupled with superconducting wires, in which each site is described by a modified, spatially discrete version of the one-dimensional sine-Gordon equation. Both methods are consistent and conditionally stable.

As an important part of the methods introduced, the work provides consistent, discrete schemes for the local energies of the links and the total energy of the system. These schemes are such that the discrete rates of change of total energy of the system are consistent approximations of the corresponding continuous rates of change, so that these techniques are ideal methods in the study of the process of nonlinear supratransmission.

As applications, we obtained diagrams of driving amplitude versus driving frequency, in order to establish the critical amplitude at which supratransmission starts. This study was carried out in the presence of the parameters under study, that is, internal and external damping, mass and generalized Josephson current. Our qualitative results are essentially similar to those obtained for the corresponding analysis of the Dirichlet case, corroborating thus the validity of our technique.

Of course, several directions of further research still remain open. For instance, the design of computational techniques that employ symplectic methods in the analysis of the transmission of energy in discrete arrays of Josephson junctions, is a topic that merits close attention. Particularly, one may use numerical methods that split the vector field $y ^\prime$ as the sum of a conservative contribution $f (y)$ --- solved in the interval $[t , t + h / 2]$ using a symplectic method for the sine-Gordon equation \cite{Referee1,Referee2,Referee3} ---, and a linear dissipative part $g (y)$ --- solved exactly in $[t , t + h]$ with an exponential integrator, for instance). The symplectic dynamics may be solved then for the next half step, obtaining thus a second-order, splitting method that is linearly implicit, and symplectic in the unperturbed case.

\subsection*{Acknowledgments}

One of us (JEMD) wishes to express his deepest gratitude to Dr. F. J. \'{A}lvarez Rodr\'{\i}\-guez, dean of the Faculty of Science of the Universidad Aut\'{o}noma de Aguascalientes, M\'{e}xico, and to Dr. F. J. Avelar Gonz\'{a}lez, dean of the Office for Research and Graduate Studies of the same university, for uninterestedly providing the resources to produce this article. The present work represents a set of partial results under research project PIM07-2 at this university, and is dedicated to Patricia Arias Mu\~{n}oz on the occasion of her $38$th academic anniversary.

\end{document}